\documentclass[a4paper,12pt]{amsart}

\usepackage[utf8]{inputenc}
\usepackage[T1]{fontenc}
\usepackage{amsmath,amssymb,amsfonts,amsthm,mathrsfs}
\usepackage{geometry}
\usepackage{graphicx}
\usepackage{float}
\usepackage[english]{babel}
\usepackage{tikz}
\usetikzlibrary{arrows}
\usetikzlibrary{decorations.markings}

\theoremstyle{plain}
\newtheorem{df}{Definition}
\newtheorem{thm}{Theorem}
\newtheorem{lem}{Lemma}
\newtheorem{prop}{Proposition}
\newtheorem*{ex}{Example}

\newtheorem{cor}{Corollary}
\newtheorem*{thm*}{Theorem}
\newcommand{\mot}{}
\newtheorem*{thmref_interne}{\mot{}}
\newenvironment{thmref}[2]{
	\renewcommand{\mot}{#1 #2}
	\begin{thmref_interne}}
	{\end{thmref_interne}
}

\newcommand\mb{\mathbb}
\newcommand\mc{\mathcal}

\newcommand\mr{\mathrm}
\newcommand\ms{\mathscr}

\renewcommand\emph[1]{\textup{\textbf{#1}}}

\title{
Irrational series II\\ Summation by packages
}

\author{Olivier Thom}
\date{}

\begin{document}
\maketitle

\begin{abstract}
Discrete sums of exponentials $g(w) = \sum a_{\beta} \mr{e}^{\beta w}$ with positive exponents may converge not normally in neighborhoods $H$ of $-\infty$ which do not contain half-planes.
We study different notions of convergence for these series and in particular the intuitive notion of summation by packages.

Indeed, joining in packages the terms in the sum $g(w)$ whose exponents are close together, and summing first inside each package may result in massive cancellations.
We show that discrete sums $g(w)$ which are bounded in what we call logarithmic neighborhoods can always be summated by packages.
\end{abstract}

\tableofcontents
\newpage

\section{Introduction}
\subsection{Context}

This article is a sequence to \cite{thom_irrational1}.
In these articles, we want to study the class of holomorphic functions which can be written $g(w) = \sum_{\beta \in R} a_{\beta} \mr{e}^{\beta w}$, where $R$ is a closed discrete subset of $\mathbb{R}^+$, mainly $R=R_{\alpha}=\mathbb{N} + \alpha^{-1} \mathbb{N}$ for some $\alpha^{-1}\in \mathbb{R}^+$ ; we will call these irrational series.
Note that if $z=\mr{e}^w$, this function can be written $\tilde{g}(z) = \sum_{\beta \in R} a_{\beta} z^{\beta}$ ; this last function is multi-valuated and we prefer to consider it as a function of the variable $w$.

These series appear naturally in several contexts, for example in problems whith small divisors.
If $\tilde{f}(z) = \mr{e}^{2i\pi \alpha}z + \sum b_n z^n$ is a germ of diffeomorphism with irrational rotation number $\alpha$, we can consider the change of variable $z=\mr{e}^w$ as before, which gives a germ $f(w) = w + 2i\pi \alpha + \sum c_n \mr{e}^{nw}$ at $w=-\infty$.
The solutions $g(w)$ of the equation $g\circ f = g$ can be written $g(w) = \sum_{\beta\in R_{\alpha}} a_{\beta} \mr{e}^{\beta w}$ for $R_{\alpha} = \mathbb{N} + \alpha^{-1} \mathbb{N}$.
This is a consequence of a result of Écalle stated in \cite{ecalle_small_denominators} ; the first article \cite{thom_irrational1} can be used to give an alternative proof of this fact, as explained in its introduction.

Normal convergence is too strong for these series.
For example, 
\[
\frac{\mr{e}^{(\beta+h)w}-\mr{e}^{\beta w}}{h} \underset{h\to 0}{\longrightarrow} \frac{d}{dt} (\mr{e}^{tw})\vert_{t=\beta} = w\mr{e}^{\beta w},
\]
so that coefficients can be arbitrarily large while the function remains bounded.
Finding and comparing other notions of convergence for these series is the purpose of this article.

The weakest possible notion of convergence for a formal series $\hat{g}=\sum a_{\beta} \mr{e}^{\beta w}$ is the existence of a neighborhood $H$ of $-\infty$ and a holomorphic function $g\in \mc{O}(H)$ whose asymptotic development at $-\infty$ is $\hat{g}$.
In general, such a function $g$ is not guaranteed to be unique, and this notion depends on the shape of the neighborhood $H$.

Another notion of convergence is Borel resummation : the Borel transform of $\hat{g}$ is $\mc{B}\hat{g} = \hat{D} = \sum a_{\beta} \delta_{\beta}$ (it is a priori only a formal sum of diracs).
If we can give a sense to the evaluation of the formal distribution $\hat{D}$ on the function $e_w(t) = \mr{e}^{tw}$, then the Laplace transform of $\hat{D}$ is $\mc{L}\hat{D}(w) = \langle \hat{D}, e_w \rangle$, and the function $\mc{L}\hat{D}$ is by definition the Borel resummation of $\hat{g}$, wherever it converges.

It turns out that when $\hat{g}$ is the asymptotic development of some function $g\in \mc{O}(H)$ (for $H$ wide enough), then $\mc{B}\hat{g}$ is given by the Laplace transform $\mc{L}g$, interpreted as a hyperfunction supported by $\mathbb{R}^+$, and $g$ is indeed the Laplace transform of its Borel transform, as we explained in \cite{thom_irrational1}.

Note that for $g(w) = \frac{1}{h}( \mr{e}^{(\beta+h)w} - \mr{e}^{\beta w})$, then $\mc{B}g = \frac{1}{h}( \delta_{\beta+h} - \delta_{\beta})$ can be seen as a discrete derivative (an approximation of $-\delta'_{\beta}$).
We can obtain higher-order discrete derivatives, consider for example the function
\[
\frac{\mr{e}^{(\beta+h)w}-2\mr{e}^{\beta w} + \mr{e}^{(\beta-h)w}}{h^2} \underset{h\to 0}{\rightarrow} w^2\mr{e}^{\beta w};
\]
the Borel transform $\frac{1}{h^2}(\delta_{\beta+h} - 2 \delta_{\beta} + \delta_{\beta-h})$ of this function can be interpreted as a discrete derivation of order 2.

For a series $g$ supported on $R_{\alpha} = \mathbb{N} + \alpha^{-1} \mathbb{N}$, the density of points in $R_{\alpha}$ close to a point $\beta\in \mathbb{R}^+$ grows linearly in $\beta$.
With $n$ points close together, we can cook up an approximation of the derivative of order $n$ so it seems reasonable to expect that the order of the distribution $\mc{B}g$ around the point $\beta\in \mathbb{R}^+$ grows linearly.
In \cite{thom_irrational1}, we proved that conditions of growth for the distribution $\mc{B}g$ corresponds to the shape of the neighborhood $H$ where $g$ is bounded.
The condition for the order of $\mc{B}g$ to grow linearly corresponds for the neighborhood $H$ to be a logarithmic neighborhood (see the body of the article for the definition).

I should mention that discrete derivatives have been used and studied for a long time in numerical analysis, but I didn't find anything in these works regarding the problems treated here.

One word about notation : we will write $ \langle D, \varphi \rangle = D(\varphi)$ for the evaluation of a distribution $D$ (a sum of diracs, or a hyperfunction) on a test function $\varphi$.
When $\varphi$ is given explicitely as a function of the variable $t$ or $p$, we will indulge in the erroneous notation $ \langle D, \varphi(t) \rangle$ : for example $ \langle D, \mr{e}^{tw} \rangle$ is equal to $D(e_w)$ where $e_w(t) = \mr{e}^{tw}$.
The variable of the test functions will always be written $t$ when it is real, or $p$ when complex.

\subsection{Results}

In this article, we will begin in section \ref{sec_definitions} by introducing irrational series and their various notions of convergence.
All these notions of convergence are equivalent for irrational series, but we try to define and study them as independently as possible, to simplify potential generalizations.
In the context of irrational series, the most complicated results of this part were proven in \cite{thom_irrational1}.

The notion of convergence studied are evanescent summation, summation by packages and diagonal extraction of integration by parts (DIPP for short).
DIPP is the most well-behaved (in fact we prove this one first, the other follow from it), but it is also the least practical.
Summation by packages is the most intuitive one, but it leads to some non-trivial considerations. We will introduce two flavours, strong and weak summation by packages, which are not a priori equivalent (although one cannot help but think that they should be, at least for irrational series).
Evanescent summation can be considered a more practical reformulation of DIPP, which is easily shown to be equivalent to weak summation by packages.

In the rest of the article, we try to find a more explicit form for the summation by packages, that is, with simple packages.
The packages we want to consider are Vandermonde packages, associated to Vandermonde distributions.
The Vandermonde distribution $\Delta_{R_n}$ associated to a finite set $R_n$ of cardinality $N_n+1$ is the only sum of diracs $\Delta_{R_n} = \sum_{\beta\in R_n} b_{\beta} \delta_{\beta}$ which approximates the $N_n$-th derivative (which amount to solving a Vandermonde system, hence the name\:; see the definition in the article).
The corresponding Vandermonde package is $\langle \Delta_{R_n}, \mr{e}^{tw} \rangle = \sum_{\beta\in R_n} b_{\beta} \mr{e}^{\beta w}$.

The main theorem of this article is

\begin{thmref}{Theorem}{\ref{thm_summation_by_packages}}
Consider an irrational series $g(w) = \sum_{\beta\in R} a_{\beta} \mr{e}^{\beta w}$ supported on a closed discrete set $R\subset \mathbb{R}^+$, defined and bounded in a logarithmic half-plane $H_{a,k}$.
We will suppose that $R$ has linear density : there exist constants $\mu,\nu$ such that for any interval $I$ of length $L$, $\# (R\cap I) \leq \mu\: L\: \mr{sup}(I) + \nu\: L$.

There exist finite sets $R_n$ with $\mr{inf}(R_n) = \beta_n$, $\#R_n = N_n+1$, coefficients $b_n\in \mathbb{C}$, a radius $r\in ]0,1[$ and a constant $c\in \mathbb{R}$ such that 
\[
\begin{aligned}
&g(w) = \sum_n b_n \langle \Delta_{R_n}, \mr{e}^{tw} \rangle,\\
& \sum_n |b_n| r^{\beta_n} < \infty,\\
& N_n \leq (\mu + 2k) \beta_n + c.
\end{aligned}
\]
\end{thmref}

These conditions imply that the series $g(w) = \sum b_n \langle \Delta_{R_n}, \mr{e}^{tw} \rangle$ converges by packages in a logarithmic neighborhood $H_{a',k'}$ where $k' = \mu + 2k$.
Since $k'>k$, this neighborhood is smaller than $H_{a,k}$.
See section \ref{sec_better_decompositions} for a discussion regarding the difficulties to obtain a sum that converges in $H_{a,k}$ and not on some smaller $H_{a',k'}$.

In section \ref{sec_distributions} we study discrete distributions and in particular decompositions $D = \sum b_n \Delta_{R_n}$.
The section \ref{sec_summation} contains the proof of Theorem \ref{thm_summation_by_packages} and some applications.

Appendix \ref{appendix_logarithmic} proves the equivalence between different definitions of logarithmic neighborhoods.

\subsection{Relation with transseries}

Irrational series are a particular case of transseries (sometimes called generalized power series), which are sums of transmonomials (for example, $(\mr{log}(z))^{\alpha}z^{\beta}\mr{e}^{\gamma z}$ is such a transmonomial).
Transseries seem to have gained interest thanks to Dulac's theorem, who found that the monodromy $f: (\mathbb{R}^+,0) \rightarrow (\mathbb{R}^+,0)$ of a polycycle for some vector field is a transseries.
More precisely, the assymptotic development $\hat{f}$ of $f(x)$ as $x\to 0$ is a transseries ; this detail has its importance since convergence for transseries is anything but a simple matter, and information on the formal series $\hat{f}$ is not easily converted to informations for the function $f$.
In fact, Dulac's work originally overlooked these difficulties, and we needed to wait for the works of Ilyashenko \cite{ilyashenko_limit_cycles} and Écalle, Martinet, Moussu, Ramis \cite{emmr_cycles_limites} to complete the proof.
See \cite{ilyashenko_centennial} for more details on this story.

Studies about transseries seem to focus mainly on the formal transseries and not their convergence in some complex domain.
Convergence problems are not trivial, even for the simplest transseries $\hat{f}(z) = \sum_{\beta\in R} a_{\beta} z^{\beta} = \sum_{\beta} a_{\beta} \mr{e}^{\beta w}$, where $R$ is a closed discrete subset of $\mathbb{R}^+$.
Informally, the questions we are interested in are\::
\begin{itemize}
\item[-] Is $\hat{f}$ the formal development (or assymptotic development) of a function $f$\:?
\item[-] Is the application $f \mapsto \hat{f}$ injective ?
\item[-] Can we compute $f$ from $\hat{f}$ (explicitely if possible) ? Or the slightly different question: do we have a resummation procedure $\hat{f} \mapsto f$, where $f$ is some function with formal development $\hat{f}$ ?
\end{itemize}
Of course, these questions depend on which functions $f$ we consider (germs of real functions, holomorphic functions in a neighborhood of the origin, maybe multivalued...)

Questions about convergence of transseries where mainly studied by Écalle, see \cite{ecalle_growth_scale}.
The resummation procedure is that of Borel resummation (modified with the accelero-summation for more complex transseries); let us explain it on a simple example.

\begin{ex}
\label{ex_resummation}
Consider the formal sum $\hat{f}(x) = \sum_{n\geq 0} n! x^{n+1}$.
For convenience, we will consider it near the infinity in the variable $p=1/x$ as $\hat{h}(p) = \sum \frac{n!}{p^{n+1}}$, and consider $p$ as the variable in the Borel plane.
Since the variables $w$ of space and $p$ of the Borel plane play a symmetric role, the way we look at it will not change the computations.
Seen as a formal sum of hyperfunctions, $\hat{h}$ corresponds to the sum of distributions $\sum_n (-1)^n \delta_0^{(n)}$.
When applied to a general test function $\varphi$ analytic in $p$, this sum does not converge ; but it does when applied to a function of exponential growth.
The Laplace transform of $\hat{h}$ is obtained applying this distribution to $\varphi(p) = \mr{e}^{p w}$, we get a convergent function
\[
g(w) := \langle \hat{h}, \mr{e}^{pw}\rangle = \sum_{n\geq 0} w^n = \frac{1}{1-w}.
\]
A priori, $g$ is only defined on $\{ |w| < 1\}$, but we can look at its analytic continuation to see it as a germ at infinity, written
\[
g(w) = \frac{-w^{-1}}{1-w^{-1}} = - \sum_{n\geq 1} w^{-n}.
\]
This sum converges normally in the set $\{\mr{Re}(w) < -a \}$ for any $a>1$, so we can consider its Laplace transform $h = \mc{L}g$.
If $\hat{h}$ were convergent (with exponential growth), $g$ would be defined on the whole of $\mathbb{C}$ and we would have $h = \hat{h}$.
The Borel resummation of $\hat{f}(x)$ is by definition $h(1/x)$ (wherever it is defined).

The Borel transform of $\hat{h}$ is $\hat{g} = \mc{B}\hat{h}$, it is obtained considering $\hat{h}$ as a formal sum of hyperfunctions and applying each one to the test function $\varphi(p) = \mr{e}^{pw}$.
The Borel resummation of $\hat{h}$ is obtain by looking at the analytic continuation $g$ of $\hat{g}$ (when $\hat{g}$ converges) and taking its Laplace transform $h=\mc{L}g$ (when $g$ has exponential growth).
\end{ex}

This is the usual resummation procedure, and we can show that the result does converge for some classes of transseries.
The question of the injectivity of $f \mapsto \hat{f}$ depends a lot on the subset $U\subset \mathbb{C}$ where $f$ is defined.
A classical argument to deal with this question is to consider an injective mapping $\Phi: \mathbb{H} \rightarrow U$, so that $f\circ \Phi$ is defined on $\mathbb{H} = \{\mr{Re}(w) < 0\}$, and use Phragmén-Lindelöf principle for $f\circ \Phi$ (see \cite{ilyashenko_centennial}), or classical Borel-Laplace transform on $\mathbb{H}$ (see \cite[Propositions 3.11 and 3.15]{ecalle_small_denominators}).

\subsection{Irrational series in the works of Écalle}

The specific case of transseries $\hat{f}(z) = \sum_{\beta\in R} a_{\beta} z^{\beta}$ (or rather, $\hat{g}(w) = \sum a_{\beta} \mr{e}^{\beta w}$) already appeared in the works of Écalle, namely in \cite{ecalle_small_denominators}.
Although he only considered series indexed by a semigroup $R$, most of his results should easily be adaptable for a general closed discrete $R$ with linear density.
Let us explain his definitions here, and then compare his results with those explained here to better understand what was already known to him (or might have been), and what is certainly new.

He introduced in this setting the compensators $z^{\beta_0,\ldots,\beta_n}$, which are the same as what we defined as Vandermonde packages
\[
z^{\beta_0,\ldots,\beta_n} = \langle \Delta_{\{\beta_0,\ldots,\beta_n\}}, z^t \rangle = \langle \Delta_{\{\beta_0,\ldots,\beta_n\}}, \mr{e}^{tw} \rangle.
\]
In his words, logarithmic neighborhoods $H$ become spiralling neighborhoods of the origin (he preferred to look at functions $g(w)$ as multivalued functions $f(z)$ defined near the origin).
Écalle defines seriable functions $\mr{Ser}(R,H)$ in the logarithmic neighborhood $H$ as the closure in $\mc{O}(H)$ of finite sums $\sum_{\beta\in R} a_{\beta} \mr{e}^{\beta w}$ for uniform convergence in $H$, and $\mr{Ser}(H) = \cup_R \mr{Ser}(R,H)$.
He also introduced compensable functions $\mr{Comp}(R,H)$ which are normally convergent sums of compensators $z^{\beta_0,\ldots,\beta_n}$, with $\beta_i\in R$, and $\mr{Comp}(H) = \cup_R \mr{Comp}(R,H)$.
In our terms, $\mr{Comp}(H)$ correspond to functions which are summable by packages, with Vandermonde packages.
In the present work, we only work with series $\sum a_{\beta} \mr{e}^{\beta w}$, while Écalle works with more general series $\sum a_{k,\beta} z^k\mr{e}^{\beta w}$ ; our choice is meant to avoid unnecessary complications, studying both types of series is strictly equivalent as we explain in section \ref{sec_degenerated}, but working with multisets from the beginning may obfuscate the intuition that we hope to convey.

The results of Écalle with respect to this matter are Propositions 3.11 and 3.15 (still in \cite{ecalle_small_denominators}).
These show that the formal development $g \mapsto \hat{g}$ is injective on $\mr{Ser}(H)$, and that in this case, a variant of the Borel resummation allows to recover $g$ from $\hat{g}$.

It is still unclear to me what is the actual scope of the proof of these propositions.
The argument "changing variable to obtain a function defined on $\mb{H}$ and using the classical Borel-Laplace transform" seems to be a good strategy in general, but the lack of details in the proofs didn't help to understand what could easily be deduced from it (Laplace transform depends non-trivially on changes of coordinates).

Some results which weren't clearly stated there, might have already been understood by Écalle, but are a consequence of our works \cite{thom_irrational1} and [this article] :
\begin{enumerate}
\item $\mr{Ser}(H)$ is exactly the set of bounded holomorphic functions in $H$ whose formal development is a formal sum $\sum a_{k,\beta} w^k\mr{e}^{\beta w}$. These are the irrational series defined here when monomials have only exponentials.
\item Écalle's results also work when $R$ is not a semigroup, but a closed discrete subset of $\mathbb{R}^+$.
\end{enumerate}
What is certainly new :
\begin{enumerate}
\item Introducing new summation methods and generalizing compensation as summation by packages.
\item $\cup_H\mr{Ser}(H) = \cup_H\mr{Comp}(H)$ : this is a consequence of our Theorem \ref{thm_summation_by_packages}.
\end{enumerate}

\subsection{Conclusion}

One question that has not been answered in this paper is whether it was really necessary to introduce a new name ("irrational series") to stand for the objects studied : after all, we saw that they are exactly seriable functions in this context.
Let us take a step back and see how the present works can help understand transseries.

Problems about convergence of transseries are quite complex, between divergence, convergence by packages and resummation.
The idea I hope to convey in this article is that the notion of convergence by packages introduced here is fundamentally different than Borel resummation.
With Borel resummation, we give a value to a divergent series (see Example \ref{ex_resummation}), whereas series which converge by packages should be seen as convergent series (this should be clear from Theorem \ref{thm_summation_by_packages}).
It seems worth studying first the simplest transseries, of the form $g(w) = \sum_{\beta\in R} a_{\beta} \mr{e}^{\beta w}$, when $R$ is a closed discrete subset of $\mathbb{R}^+$, and which are already convergent in some sense.

The simplest ones are those indexed by $R=\mathbb{N}$, that is, usual power series (also called entire series) : they converge normally on a straight half-plane.
Series indexed by any subset $R$ which also converge normally on a straight half-plane can be studied rather similarly (apart from the fact that entire series are $2i\pi$-periodic)\:: they correspond to holomorphic functions which are bounded on a half-plane and whose formal development is indexed by $R$.
Series indexed by $R = \mathbb{N} + \alpha^{-1} \mathbb{N}$ for some rational $\alpha$ are Puiseux series, they can be reduced to entire series through a change in the variable $w$.

Next on the list are series indexed by $R = \mathbb{N} + \alpha^{-1} \mathbb{N}$ with an irrational $\alpha$, which certainly deserve the name of irrational series.
This class of functions appears naturally in problems with small denominators, are easy to study through Laplace transform or summation by packages, and this is why I think they deserve a name of their own.
Series indexed by any subset $R$ which are bounded in a logarithmic half-plane are not more difficult to study and we did not give them a different name in this paper.

More general series $g(w)$ can be studied in a similar manner.
One property which was key in all of this is the property for the Laplace transform $\mc{L}g$ (seen as a hyperfunction) to have locally finite order.
As explained in \cite{thom_irrational1}, any growth property for $\mc{L}g$ is in fact a property on the shape of the neighborhood $H$ of $-\infty$ on which $g$ is bounded.
If $\mc{L}g$ has locally finite order, we should be able to do a DIPP, and the "value" of the hyperfunction $\mc{L}g$ at a point $t\in \mathbb{R}^+$ only depends on the coefficients $a_{\beta}$ with $\beta\leq t$ : this is what characterizes seriability in the terminology of Écalle.
I think we can use the term "seriability" more generally for this kind of series.

We can then wonder whether series indexed by a closed discrete subset are always seriable (that is, bounded in a neighborhood $H$ which satisfies the property of locally finite order).
As it happens, finding the biggest neighborhood $H$ on which a series is bounded, provided it is bounded in some neighborhood, is not an easy task.
Note that even for series indexed by $\mathbb{N} + \alpha^{-1} \mathbb{N}$, we did not prove that if it converges in some neighborhood, then it converges in a logarithmic neighborhood.
In fact, even for series indexed by $\mathbb{N}$, we had to use $2i\pi$-periodicity to conclude that they converges normally.
The tools developped here might need further polishing before being able to solve this kind of problems.

\section{Definitions}
\label{sec_definitions}
\subsection{Irrational series}

A formal irrational series in one variable $z$ is a sum $\hat{f}(z) = \sum_{\beta\in R} a_{\beta} z^{\beta}$ where $R$ is a closed discrete subset of $\mathbb{R}^+$, and $a_{\beta}\in \mathbb{C}$ are complex coefficients.

The terms $z^{\beta}$ have non-trivial monodromy as $z$ circles around the origin, so we study these series by lifting them to the universal covering of $\mb{D}^*$.
Following \cite{thom_irrational1}, we say that a subset $H\subset \mathbb{C}$ is a neighborhood of $-\infty$ if for each $\theta< \pi/2$, it contains a left-handed cone of opening $2 \theta$ symmetric about $\mathbb{R}^-$.
A neighborhood $H$ is called logarithmic if it is (or, contains one) of the form 
\[
H_{a,k} = \{ x+iy \in \mathbb{C} \:|\: x + \frac{k}{2}\:\mr{log}(x^2+y^2) < a \}
\]
for some $a\in \mathbb{R}$ and $k>0$.
We gave in \cite[Appendix A]{thom_irrational1} several definitions of logarithmic neighborhoods, this is one of them.

The covering application $H \rightarrow \mb{D}^*$ is just $w\in H \mapsto z=\mr{e}^w$, so formal irrational series correspond to sums $\hat{g}(w) = \sum_{\beta\in R} a_{\beta}\mr{e}^{\beta w}$.
We say that a holomorphic function $g\in \mc{O}(H)$ admits $\hat{g}$ as formal development if for each $\beta_0\in R$, 
\[
g(w) - \sum_{\beta < \beta_0} a_{\beta}\mr{e}^{\beta w} = O(\mr{e}^{\beta_0 w})
\]
as $\mr{Re}(w) \rightarrow -\infty$.

\begin{df}
We say that the formal irrational series $\hat{g}(w) = \sum_{\beta\in R} a_{\beta} \mr{e}^{\beta w}$ converges normally on the neighborhood $H$ if $\sum_{\beta\in R} |a_{\beta}| \mr{e}^{\beta r} < \infty$, where $r = \mr{sup}_{w\in H} \mr{Re}(w)$.
\end{df}

\begin{ex}
The exponent set $R$ we are mainly interested in is $R = R_{\alpha} = \mathbb{N} + \alpha^{-1} \mathbb{N}$, and consists of points $\beta = p + \frac{q}{\alpha}$.
Note that if $\mr{Re}(w) <0$,
\[
\begin{aligned}
\sum_{\beta\in R_{\alpha}} \mr{e}^{\beta w} &= \sum_{q=0}^\infty \sum_{p=0}^\infty \mr{e}^{\frac{q}{\alpha}w}\mr{e}^{pw}\\
&= \sum_{q=0}^{\infty} \frac{\mr{e}^{\frac{q}{\alpha}w}}{1-\mr{e}^{w}}\\
&= \frac{1}{(1-\mr{e}^{\frac{w}{\alpha}})(1-\mr{e}^{w})}.
\end{aligned}
\]
It follows that the domain of normal convergence of a series $\sum_{\beta\in R_{\alpha}} a_{\beta} \mr{e}^{\beta w}$ is the half-plane
\[
\left\{ \mr{Re}(w) < \mr{lim}\:\mr{inf}\: \frac{-1}{\beta}\mr{log}|a_{\beta}| \right\},
\]
where we understand $\mr{log}(0)$ as beeing $-\infty$.
\end{ex}

\begin{df}
We say that a formal irrational series $\hat{g}$ converges in the logarithmic neighborhood $H$ if there exists a bounded holomorphic function $g$ on $H$ which admits $\hat{g}$ as formal development.

We call irrational series a bounded holomorphic function $g$ on a logarithmic half-plane $H$ whose formal development is a formal irrational series.
\end{df}

\subsection{Convergence as a hyperfunction}

Note that we can write formally $\sum_{\beta\in R} a_{\beta} \mr{e}^{\beta w} = \langle D, \mr{e}^{tw} \rangle$ where $D=\sum_{\beta\in R} a_{\beta} \delta_{\beta}$ is to be understood informally as a distribution of the variable $t$.

The correct context for this point of view is that of hyperfunctions in the sense of Sato (cf. \cite{sato_hyperfunctions1}, \cite{sato_hyperfunctions2}) (see also \cite{morimoto_hyperfunctions} for an introductory book): the distribution $\delta_{\beta}$ corresponds to the hyperfunction $\frac{1}{p-\beta}\in \mc{O}(\mathbb{C}\setminus \mathbb{R}^+) / \mc{O}(\mathbb{C})$.
This point of view has been studied in \cite{thom_irrational1}, where we proved that the Laplace transform realises an equivalence between bounded holomorphic functions in a logarithmic half-plane $H$ and some class $\ms{H}$ of hyperfunctions in $\mathbb{C}$ supported by $\mathbb{R}^+$.

From this equivalence, if the formal irrational series $\hat{g}(w) = \sum_{\beta\in R} a_{\beta} \mr{e}^{\beta w}$ converges in $H$ (denote by $g$ the sum), the hyperfunction $D(p)=\sum_{\beta\in R} \frac{a_{\beta}}{p-\beta}$ converges as a hyperfunction in $\ms{H}$, and is equal to the Laplace transform of $g$ : we have $D = \mc{L}g$.

Conversely, if there is a hyperfunction $D\in \ms{H}$ supported by $R$ which is equal to $\frac{a_{\beta}}{p-\beta}$ near each $\beta\in R$, the inverse Laplace transform $g = \mc{L}^{-1}D$ is a function in a logarithmic half-plane $H$ whose formal development is $\sum_{\beta\in R} a_{\beta} \mr{e}^{\beta w}$.

The sum $\sum_{\beta\in R} \frac{a_{\beta}}{p-\beta}$ does not converge normally, but we have some summation formulas which we will explain in the next paragraphs (see \cite{thom_irrational1} for more).

\subsection{Evanescent summation}

We say that a finite sum $\tilde{S}_ng=\sum b_i \mr{e}^{\beta_i w}$ is a $n$-th evanescent partial sum of a formal irrational series $\hat{g}$ if $\hat{g} - \tilde{S}_ng = o(\mr{e}^{nw})$ formally.
We say that a bounded holomorphic function $g$ in a logarithmic neighborhood $H$ is an evanescent summation of the formal series $\hat{g}$ if there exists for each $n\in \mathbb{N}$ an evanescent partial sum $\tilde{S}_ng$ such that the sequence $(\tilde{S}_ng)_n$ converges uniformly to $g$ in $H$.

In \cite{thom_irrational1}, we proved that each irrational series $g$ admits an evanescent summation\:; quite importantly, the summation that we found can be explicitely written in function of the formal development of $g$.

\begin{thm}
\label{thm_evanescent}
There exists an explicit way to associate to each formal irrational series $\hat{g}$ a sequence of evanescent partial sums $\tilde{S}_n \hat{g}$ so that whenever $g$ is an irrational series with formal development $\hat{g}$, the sequence $\tilde{S}_n\hat{g}$ converges uniformly to $g$ in a neighborhood of $-\infty$.
\end{thm}

This is \cite[Theorem 4]{thom_irrational1}.
An important consequence of this construction is the injectivity of the formal development.
As explained in the introduction, this can also be proved by more classical tools.

\begin{cor}
\label{cor_injectivity}
Two irrational series $g_1, g_2$ with the same formal development are equal.
\end{cor}

\begin{proof}
Each $g_i$ is the limit of its evanescent partial sums $\tilde{S}_ng_i$, and these partial sums only depend on the formal development $\hat{g}_1 = \hat{g}_2$.
We obtain the equality of $g_1$ and $g_2$ in some neighborhood of $-\infty$ and conclude by analytic continuation.
\end{proof}

\subsection{Summation by packages}

Fix a closed discrete subset $R\subset \mathbb{R}^+$ and a formal irrational series $\hat{g}$ supported on $R$.

Define a system of packages supported by $R$ as a family $\ms{P} = \{ P_n, n\in \mathbb{N}\}$ where each $P_n$ is a finite subset of $R$ and $R = \bigcup_{n\in \mathbb{N}} P_n$ (some packages can overlap).

We say that a formal sum $\hat{g}_{\ms{P}}(w) = \sum_{n\in \mathbb{N}}\sum_{\beta\in P_n} a_{n,\beta} \mr{e}^{\beta w}$ is a decomposition of $\hat{g}$ along $\ms{P}$ if for each $\beta\in R$, $\sum_{n\in \mathbb{N}} a_{n,\beta}=a_{\beta}$.
Such a formal sum is said to converge normally by packages on a subset $H\subset \mathbb{C}$ if
\[
\sum_{n\in \mathbb{N}} \sup_{w\in H} \left| \sum_{\beta\in P_n} a_{n,\beta} \mr{e}^{\beta w} \right| < \infty.
\]
The formal series $\hat{g}$ will be said to converges normally by packages on $H$ ("converges by packages" for short) if there exists a decomposition of $\hat{g}$ along $\ms{P}$ which converges by packages on $H$.
We say that $\hat{g}$ converges (strongly) by packages (or that $\hat{g}$ can be summated by packages) if it converges normally by packages on some logarithmic neighborhood $H$ of $-\infty$.

Let us introduce a weaker version of convergence by packages : we say that the formal irrational series $\hat{g}$ supported by $R$ converges weakly by packages if there is a closed discrete subset $\tilde{R }\supset R$, a system of packages $\ms{P}$ supported by $\tilde{R }$ and a decomposition $\hat{g}_{\ms{P}}$ of $\hat{g}$ along $\ms{P}$ which converges normally by packages on a logarithmic neighborhood.

The notion of evanescent summation is equivalent to that of weak summation by packages.

\begin{prop}
A formal irrational series $\hat{g}$ admits a weak summation by packages if and only if it admits an evanescent summation.
\end{prop}

\begin{proof}
If $\hat{g}$ admits a weak summation by packages $\hat{g}_{\ms{P}}$, then the finite sums $\tilde{S}_{m}g = \sum_{i\leq n} \sum_{\beta\in P_i} a_{i,\beta}\mr{e}^{\beta w}$ provide evanescent partial sums of $\hat{g}$ which converge uniformly to $\hat{g}_{\ms{P}}$ (the order $m$ up to which $\hat{g}$ and $\tilde{S}_mg$ coincide is not equal to $n$, but we only need it to tend to infinity when $n\to \infty$).

Conversely, if $(\tilde{S}_ng)_n$ is a sequence of evanescent partial sums converging to a function $g$ in some logarithmic half-plane $H$, we can define $A_n(w) = \tilde{S}_ng(w) - \tilde{S}_{n-1}g(w)$ for $n\geq 1$, $A_0(w) = \tilde{S}_0(w)$ and $P_n = \mr{supp}(\mc{L}A_n)$.
By Cauchy's criteria, the sum $\sum_n A_n$ converges normally, and its sum is $\sum_n A_n(w) = g(w)$.
This is to say that the formal series $\hat{g}$ converges weakly along the system of packages $\ms{P} = \{ P_n \}_{n}$ and its sum is $\hat{g}_{\ms{P}}(w) = g(w)$.
\end{proof}

In particular, Theorem \ref{thm_evanescent} can be restated :

\begin{thm}
\label{thm_weak_summation}
An irrational series can be weakly summated by packages.
\end{thm}

\subsection{Diagonal integration by parts}

For any formal series $\hat{g}(w) = \sum_{\beta\in R} a_{\beta}\mr{e}^{\beta w}$, we can once again introduce the distribution $D=\sum_{\beta\in R} a_{\beta} \delta_{\beta}$.
When the series converges normally to the function $g$, we have 
\[
g(w) = \langle D(t),\mr{e}^{t w} \rangle = \int_{t\in \mathbb{R}^+} D(t)\mr{e}^{t w}dt.
\]

Introduce for any distribution supported on $\mathbb{R}^+$ the operator $I$ computing the primitive which is supported on $\mathbb{R}^+$: if $\varphi$ is a function supported on $\mathbb{R}^+$, then $I \varphi(t) = \int_{s=0}^{t} \varphi(s)ds$.
For any $t_1\in \mathbb{R}^+$, the integration by parts on $[t_1,\infty[$ gives formally
\[
\int_{t\in \mathbb{R}^+} D(t)\mr{e}^{tw}dt = \int_{t=0}^{t_1}D(t)\mr{e}^{tw}dt - (ID)(t_1)\mr{e}^{t_1w} - \int_{t=t_1}^{\infty} (ID)(t)w\mr{e}^{tw}dt.
\]
For any weakly increasing sequence $(t_n)$ such that $t_0=0$ and $t_n\to \infty$, we can introduce the diagonal extraction of integration by parts (diagonal integration by parts or DIPP for short)
\[
\begin{aligned}
I_{(t_n)}^{\Delta}(D,\mr{e}^{tw}) &:= \sum_{n\geq 0} (-1)^n \int_{t=t_n}^{t_{n+1}} (I^nD)(t)w^n\mr{e}^{tw}dt\\
&\quad + \sum_{n\geq 1} (-1)^n (I^nD)(t_n)w^{n-1}\mr{e}^{t_nw}.
\end{aligned}
\]
In order to avoid unnecessary complications, we will always assume that $t_n \notin R$ for every $n$.
We admit in this definition the cases $t_n=t_{n+1}$ for some $n$'s (several IPP at the same time) or $t_n=t_{n+1}=\ldots =\infty$ (finite number of IPPs).

\begin{df}
We say that the diagonal integration by parts $I_{(t_n)}^{\Delta}(D,\mr{e}^{tw})$ converges at the point $w$ if 
\[
\sum_{n\geq 0} \int_{t=t_n}^{t_{n+1}} |(I^nD)(t)| |w^n\mr{e}^{tw}|dt + \sum_{n\geq 1} |(I^nD)(t_n)| |w^{n-1}\mr{e}^{t_nw}| < \infty.
\]
We say that it converges on some subset $H\subset \mathbb{C}$ if it converges at every point $w\in H$.

If there exists $a\in \mathbb{R}$ such that
\[
\sum_{n\geq 0} \int_{t=t_n}^{t_{n+1}}|(I^nD)(t)|\mr{e}^{at} dt + \sum_{n\geq 1} |(I^nD)(t_n)|\mr{e}^{at_n} < \infty,
\]
we say that the DIPP $I^{\Delta}_{(t_n)}(D,\mr{e}^{wt})$ is bounded by $a$.
\end{df}

\begin{df}
Let us call a sequence $(t_n)$ admissible if it increases weakly, satisfies $t_0=0$, $t_n \rightarrow \infty$, and grows superlinearly : $\displaystyle \mr{liminf}\: \frac{t_n}{n}>0$.
\end{df}

Let us begin by some remarks on these definitions :
\begin{itemize}
\item[-] the primitive $ID$ is a locally bounded function ; the modulus $|(I^nD)(t)|$ in the definition is the modulus of this function and its primitives.
\item[-] DIPP only depends on the asymptotic behaviour of the sequence $(t_n)$ : if $t_n = s_n$ for every $n\geq N$, then $I^{\Delta}_{(t_n)}(D,\mr{e}^{wt})$ and $I^{\Delta}_{(s_n)}(D,\mr{e}^{wt})$ converge for the same points $w$ and have the same value.
\item[-] if the sequence $(t_n)$ grows superlinearly $\frac{t_n}{n}\geq \frac{1}{k}$ for $n>0$, we can use Lemma \ref{lem_cle} from the appendix : $|w^n\mr{e}^{wt}| \leq \mr{e}^{at}$ for $t\in [t_n,t_{n+1}]$ and $w\in H_{a,k}$.
\item[-] if $\frac{t_n}{n}\geq \frac{1}{k}$ and $I^{\Delta}_{(t_n)}(D,\mr{e}^{wt})$ is bounded by $a$, then it converges on $H_{a,k}$, and the sum is a bounded holomorphic function.
\end{itemize}

\begin{prop}
\label{prop_DIPP_bounded}
Suppose $(t_n)$ is an admissible sequence such that $I^{\Delta}_{(t_n)}(D,\mr{e}^{wt})$ converges at some point $w$ with $|w|\geq 1$.
Then $I^{\Delta}_{(t_n)}(D,\mr{e}^{wt})$ is bounded by some number $a$.
\end{prop}

In particular, if the DIPP converges at some point as above, it converges in a neighborhood of $-\infty$, where it defines a bounded holomorphic function.

\begin{proof}
Put $a = \mr{Re}(w)$, we have $|w^n\mr{e}^{wt}| \geq \mr{e}^{at}$, and the convergence of the DIPP at the point $w$ implies
\[
\sum_{n\geq 0} \int_{t_n}^{t_{n+1}} |(I^nD)(t)| \mr{e}^{at} dt + \sum_{n\geq 1} |(I^nD)(t_n)| \mr{e}^{at_n} < \infty
\]

\end{proof}

Let us write the details of the DIPP to prove that $\hat{g}$ is indeed the formal development of $I^{\Delta}_{(t_n)}(D,\mr{e}^{tw})$.

\begin{prop}
\label{prop_DIPP_summation}
Suppose $(t_n)$ is an admissible sequence for which $I^{\Delta}_{(t_n)}(D,\mr{e}^{wt})$ converges in the neighborhood $H$.
Write 
\[
\begin{aligned}
\widetilde{A}_n\hat{g}(w) &= (-1)^n\int_{t_n}^{t_{n+1}} (I^nD)(t)w^n\mr{e}^{tw}dt + (-1)^{n+1} (I^{n+1}D)(t_{n+1})w^n\mr{e}^{t_{n+1}w}\\
&= \langle D_n, \mr{e}^{tw} \rangle,
\end{aligned}
\]
where $D_n-\sum_{\beta\in R\cap [t_n,t_{n+1}]} a_{\beta} \delta_{\beta}$ is given by the border terms of the integration by parts : it is a distribution of order at most $n$ supported by $\{t_n,t_{n+1}\}$.

Then $I^{\Delta}_{(t_n)}(D,\mr{e}^{tw}) = \sum_{n\in \mathbb{N}} \widetilde{A}_n\hat{g}(w)$ is a weak summation by packages of $\hat{g}$.
\end{prop}

\begin{proof}
That the sum $\sum_{n\in \mathbb{N}} \widetilde{A}_n\hat{g}$ converges is a consequence of the convergence of the DIPP $I^{\Delta}_{(t_n)}(D,\mr{e}^{wt})$. Write $g(w)$ the sum.
We need to check that $\widetilde{A}_n\hat{g}(w)$ is indeed a finite sum $\langle D_n, \mr{e}^{tw} \rangle$, and that the formal development of $g$ is equal to $\hat{g}$.
Introduce $\widetilde{S}_ng(w) = \sum_{k\leq n} \widetilde{A}_n\hat{g}(w)$ (these will be the evanescent partial sums).
The remainder $g - \widetilde{S}_ng$ is obviously a $O(\mr{e}^{t_{n+1}w})$ and we will first check that $\widetilde{S}_ng(w) = \sum_{\beta\leq t_{n+1}} a_{\beta} \mr{e}^{\beta w} + BT_{n+1}$ where $BT_{n+1}$ is a border term $BT_{n+1} = \sum_{k\leq n+1} b_{n+1,k}w^{k-1}\mr{e}^{t_{n+1}w}$.

If $t_k\neq \beta$ for every $k,\beta$, we can make our integrations by parts:
\[
\begin{aligned}
\sum_{\beta\leq t_{n+1}} a_{\beta}\mr{e}^{\beta w} &= \sum_{\beta\leq t_1} a_{\beta} \mr{e}^{\beta w} - \int_{t_1}^{t_{n+1}} (ID)(t)w\mr{e}^{tw}dt + (ID)(t_{n+1})\mr{e}^{t_{n+1}w} - (ID)(t_1)\mr{e}^{t_1w}\\
&= \widetilde{A}_0\hat{g}(w) - \int_{t_1}^{t_{n+1}} (ID)(t)w\mr{e}^{tw}dt + (ID)(t_{n+1})\mr{e}^{t_{n+1}w}\\
&= \widetilde{A}_0\hat{g}(w) - \int_{t_1}^{t_2} (ID)(t)w\mr{e}^{tw}dt + \int_{t_2}^{t_{n+1}} (I^2D)(t)w^2\mr{e}^{tw}dt \\
&\quad + (ID)(t_{n+1})\mr{e}^{t_{n+1}w} - (I^2D)(t_{n+1})w\mr{e}^{t_{n+1}w} + (I^2D)(t_2)w\mr{e}^{t_2w}\\
&= \widetilde{A}_0\hat{g}(w) + \widetilde{A}_1\hat{g}(w) + \int_{t_2}^{t_{n+1}} (I^2D)(t)w^2\mr{e}^{tw}dt \\
&\quad + (ID)(t_{n+1})\mr{e}^{t_{n+1}w} - (I^2D)(t_{n+1})w\mr{e}^{t_{n+1}w}\\
&= \ldots\\
&= \widetilde{A}_0\hat{g}(w) + \ldots + \widetilde{A}_{n-1}\hat{g}(w) + (-1)^{n}\int_{t_n}^{t_{n+1}} (I^nD)(t)w^n\mr{e}^{tw}dt \\
&\quad + (ID)(t_{n+1})\mr{e}^{t_{n+1}w} + \ldots + (-1)^{n+1}(I^{n}D)(t_{n+1})w^{n-1}\mr{e}^{t_{n+1}w}\\
&= \widetilde{S}_ng(w) + \sum_{k=1}^{n+1} (-1)^{k+1} (I^kD)(t_{n+1})w^{k-1}\mr{e}^{t_{n+1}w}.
\end{aligned}
\]

Thus $\widetilde{S}_ng$ is indeed an evanescent partial sum for $\hat{g}$.
To conclude, note that 
\[
\widetilde{A}_n\hat{g}(w) = \widetilde{S}_ng(w) - \widetilde{S}_{n-1}g(w) = \sum_{t_n < \beta \leq t_{n+1}} a_{\beta}\mr{e}^{\beta w} + BT_{n+1} - BT_{n}.
\]
The term $BT_{n+1} = \sum_{k\leq n+1} b_{n+1,k} w^{k-1}\mr{e}^{t_{n+1}w}$ can be interpreted as a derivation of order at most $n$ supported at $t=t_{n+1}$ since $\langle \delta^{(k)}_{\beta}, \mr{e}^{tw}\rangle = (-1)^k w^k\mr{e}^{\beta w}$.
\end{proof}

In particular (and this was not obvious), the process of DIPP does not depend on the sequence $(t_n)$ as stated in the following proposition.

\begin{prop}
Suppose $(t_n), (s_n)$ are admissible sequences.
If $I^{\Delta}_{(t_n)}(D,\mr{e}^{wt})$ and $I^{\Delta}_{(s_n)}(D,\mr{e}^{wt})$ converge in some neighborhoods $H_t, H_s$, then they are equal in $H = H_t\cap H_s$.
\end{prop}

\begin{proof}
From Proposition \ref{prop_DIPP_bounded}, the functions $g_1(w) = I^{\Delta}_{(t_n)}(D,\mr{e}^{wt}), g_2(w) = I^{\Delta}_{(s_n)}(D,\mr{e}^{wt})$ are two bounded holomorphic functions in some neighborhood $H'\subset H$ of $-\infty$.
Their development as a formal irrational series is the same (it is exactly $D$) ; by the injectivity of the formal development (Corollary \ref{cor_injectivity}), $g_1$ and $g_2$ must coincide on $H$.
\end{proof}

\begin{prop}
\label{prop_admissible}
Suppose $g$ is an irrational series defined in a logarithmic half-plane $\widetilde{H}_{a,k}$.
Then $I^{\Delta}_{(t_n)}(D,\mr{e}^{tw})$ converges for the admissible sequence $t_0=t_1=t_2=t_3=0$, $t_n = \frac{n-3.5}{k}$ if $n>3$.
\end{prop}

This result is \cite[Theorem 3]{thom_irrational1}.
If one of the points $t_n$ falls in the support of $g$, we can harmlessly choose another point $t_n$ close to it.

\section{Discrete distributions}
\label{sec_distributions}
\subsection{Vandermonde distributions}

Consider a finite subset $R = \{ \beta_0, \ldots, \beta_n \} \subset \mathbb{R}^+$.
We are concerned with distributions $D = \sum_{\beta\in R} a_{\beta} \delta_{\beta}$ of order 0 and supported on $R$.
For any such $R$, we can define the discrete distribution $D_R$ of order $n$ by the equations 
\[
\left\{
\begin{aligned}
D_R(1) &= 0\\
D_R(t) &= 0\\
&\vdots\\
D_R(t^{n}) &= 1
\end{aligned} \right.
\Leftrightarrow \left\{
\begin{aligned}
\sum a_{\beta} &=0\\
\sum \beta a_{\beta} &= 0\\
&\vdots\\
\sum \beta^{n} a_{\beta} &=1.
\end{aligned} \right.
\]

Note that $D_R$ is the solution of a Vandermonde system, hence the name "Vandermonde distribution".
We also introduce the normalized Vandermonde distribution $\Delta_R = n! D_R$, which is a discrete version of the $n$-th derivative.
We can sum up the combinatorics of these solutions in the following proposition.

\begin{prop}
The Vandermonde distribution $D_R = \sum_{\beta\in R} a_{\beta} \delta_{\beta}$ associated to the finite set $R = \{ \beta_0,\ldots, \beta_n\}$ satisfies the following properties:
\begin{enumerate}
\item $\displaystyle a_{\beta_i} = \prod_{p\neq i}\frac{1}{\beta_i-\beta_p}$,
\item for $k>n$, $\displaystyle D_r(t^k) = \sum_{p_0+\ldots+p_n=k-n} \beta_0^{p_0}\ldots \beta_n^{p_n}$.
\end{enumerate}
\end{prop}

\begin{proof}
By substituing $a_{\beta_n} = -\sum_{i=0}^{n-1} a_{\beta_i}$ into the expression $\sum \beta^k a_{\beta}$, we get :
\[
\begin{aligned}
\sum_{i=0}^{n} \beta_i^k a_{\beta_i} &= \sum_{i=0}^{n-1} a_{\beta_i}(\beta_i-\beta_n)(\beta_i^{k-1}+ \beta_i^{k-2}\beta_n + \ldots + \beta_n^{k-1})\\
&= \sum_{p=0}^{k-1} \left(\sum_{i=0}^{n-1} a_{\beta_i}(\beta_i-\beta_n) \beta_i^{k-p-1} \right) \beta_n^{p}
\end{aligned}
\]

For $k=1,\ldots,n-1$ we obtain recursively $\sum a_{\beta_i}(\beta_i-\beta_n)\beta_i^{k-1} = 0$, and for $k=n$ we get $\sum a_{\beta_i}(\beta_i-\beta_n) \beta_i^{n-1}=1$.
Thus we see that if $R_{n-1} = R\setminus \{\beta_n\}$,
\[
D_{R_{n-1}} = \sum_{i=0}^{n-1} a_{\beta_i}(\beta_i - \beta_n) \delta_{\beta_i},
\]
the expression for the coefficients $a_{\beta_i}$ follow from an easy recursion.

If $k>n$ the equation above reads 
\[
D_R(t^k) = \sum_{p=0}^{k-1} D_{R_{n-1}}(t^{k-p-1}) \beta_n^p = \sum_{p=0}^{k-n} D_{R_{n-1}}(t^{k-p-1}) \beta_n^p.
\]
Since for $R_0=\{\beta_0\}$ we have $D_{R_0}=\delta_{\beta_1}$, we see by induction that
\[
D_R(t^k) = \sum_{p_0+\ldots+p_n=k-n} \beta_0^{p_0}\ldots \beta_n^{p_n},
\]
where the sum is indexed by all families $(p_i)$ with $p_i\geq 0$ and $\sum p_i = k-n$.
\end{proof}

The following proposition justifies that $\Delta_R$ is a discrete version of the $n$-th derivative $\delta^{(n)}_{\beta_*}$ at some point $\beta_*$.

\begin{prop}
\label{prop_majoration_vandermonde}
For every real test function $\varphi\in \mc{C}^n(I,\mathbb{R})$, there exists a point $\beta_*\in [\mr{min}(R), \mr{max}(R)]$ such that 
\[
\Delta_R(\varphi) = \varphi^{(n)}(\beta_*).
\]
\end{prop}

\begin{proof}
We can order the numbers $\beta_i$ so that $\beta_0 < \ldots < \beta_n$.
Consider the polynomial $P$ of degree $n$ such that $P(\beta_i)=\varphi(\beta_i)$ for each $i=0,\ldots,n$.
Of course $\Delta_R(\varphi) = \Delta_R(P)$, but since $\Delta_R(t^k)=0$ if $k<n$ and $\Delta_R(t^n)=n!$, we also have $\Delta_R(P) = P^{(n)}(t)$ for every $t\in I$.
We can thus consider $\psi = \varphi - P$: this function satisfies $\psi(\beta_i)=0$ for every $i$, and we need to prove that $\psi^{(n)}(t)=0$ for some $t\in [\beta_0, \beta_n]$.
But we can obtain this result by applying repeatedly Rolle's Theorem\:: $\psi$ has $n+1$ zeroes, so $\psi'$ has $n$ zeroes, …, so $\psi^{(n)}$ has one zero.
\end{proof}

For complex-valued test functions, we can use the latter proposition on its real and imaginary parts to obtain the following inequalities.

\begin{cor}
\label{cor_majoration_vandermonde}
For every test function $\varphi\in \mc{C}^{n+1}(I,\mathbb{C})$ and every point $\beta_*\in I$, if $L$ is the length of $I$,
\[
|\Delta_R(\varphi)-\varphi^{(n)}(\beta_*)| \leq 2L\: \mr{sup}_I | \varphi^{(n+1)}|.
\]
Similarly, if $R_1, R_2$ are two subsets of $I$ of cardinal $n+1$, 
\[
|\Delta_{R_1}(\varphi) - \Delta_{R_2}(\varphi)| \leq 2L\: \mr{sup}_I |\varphi^{(n+1)}|.
\]
If $\varphi\in \mc{C}^n(I,\mathbb{C})$, we have 
\[
|\Delta_R(\varphi)| \leq \sqrt{2} \mr{sup}_I |\varphi^{(n)}|.
\]
\end{cor}

When trying to decompose the set $R$ into smaller sets, the following lemma will come in handy.

\begin{lem}
\label{lem_decomposition_vandermonde}
Suppose that $R=R'\cup\{\beta_1,\beta_2\}$.
Denote $R_1 = R'\cup \{\beta_1\}$ and $R_2 = R'\cup \{\beta_2\}$.
Then
\[
D_R = \frac{D_{R_1}-D_{R_2}}{\beta_1-\beta_2}.
\]
\end{lem}

\begin{proof}
We just have to check the Vandermonde system.
Denote by $D$ the distribution $D=\frac{D_{R_1}-D_{R_2}}{\beta_1-\beta_2}$ and by $n+1$ the cardinality of $R'$.
We already have $D(t^k)=0$ for $k=0,\ldots,n$; by definition, $D(t^{n+1}) = (1-1)/(\beta_1-\beta_2) = 0$.
Finally,
\[
D(t^{n+2}) = \left( \sum_{\beta\in R_1} \beta - \sum_{\beta\in R_2} \beta \right)/(\beta_1-\beta_2) = 1.
\]
\end{proof}

\subsection{Norms of distributions}

A discrete distribution $D = \sum_{\beta\in R} a_{\beta}\delta_{\beta}$ supported on $R=\{\beta_0,\ldots,\beta_n\}$ can be said to be of order $k$ if $D(t^p)=0$ for each $p < k$.
The set of distributions of order $k$ is spanned by the normalized Vandermonde distributions $\Delta_{S}$, where $S\subset R$ has cardinality $k+1$.
In fact, if $S_p = \{\beta_p, \beta_{p+1},\ldots, \beta_{p+k}\}$, the family $(\Delta_{S_p})_{p=0,\ldots,n-k}$ is a basis of distributions of order $k$.

In order to obtain other decompositions, we will also consider decompositions of $D$ on the family $(\Delta_S)$ where $S\subset R$ has cardinality $\#S \geq k+1$.

\begin{df}
Let $D$ be a discrete distribution of order $k$ supported on $R$ and $I$ an interval containing $R$.
The $k$-th functional norm of $D$ on $I$ is
\[
N_k^{fun}(D) := \mr{sup}_{\varphi} \frac{|D(\varphi)|}{\| \varphi^{(k)}\|_I},
\]
where the supremum is taken over all test functions $\varphi\in C^k(I)$ with $\varphi^{(k)}\not\equiv 0$.
\end{df}

This definition depends on $I$, but this dependance should not be significant. One could also consider test functions on the interval $\mr{Conv}(R)$ whose derivatives admit limits at the end of this interval.

To reflect decompositions of a distribution $D$ on a generating family $(\Delta_S)$ for some subsets $S\subset R$, we introduce combinatorial norms :

\begin{df}
Let $D$ be a discrete distribution of order $k$ supported on $R$ and $\rho>0$ a positive ratio.
We introduce $k$-th combinatorial norms of $D$ as
\[
N_k^{comb}(D) := \underset{D=\sum_S b_S \Delta_S}{\mr{inf}} \sum_S |b_S|,
\]
\[
N_{\geq k}^{comb}(D) := \underset{D=\sum_S b_S \Delta_S}{\mr{inf}} \sum_S \rho^{k-|S|}|b_S|,
\]
where the infimum is taken over all decompositions $D=\sum_S b_S \Delta_S$ of $D$ with respectively $\#S = k+1$ and $\#S \geq k+1$.
\end{df}

Considering the vector space $E$ with basis $(e_S)_S$ (interpreting $e_S$ as a meaningless symbol), we can consider the linear application $A : E \rightarrow {\mathbb{C}}^R$ given by $A(e_S)=\Delta_S$ for each $S$.
The space $E$ can be equipped with the norm $\| \sum b_S e_S \|_1 = \sum \rho^{k-|S|}|b_S|$\:; the combinatorial norm is the quotient norm in $E / \mr{ker}(A)$.

The parameter $\rho$ can be chosen arbitrarily, but it will be convenient to choose $\rho = L$, the length of $I$ (and we will do this choice in the following).

We have $N_k^{fun}(D) \leq N_k^{comb}(D)$ for every distribution $D$ of order $k$, this follows from Corollary \ref{cor_majoration_vandermonde} ; we also have $N_{\geq k}^{comb}(D)\leq N_{k}^{comb}(D)$.
What we want is to find good inequalities $N_{\geq k}^{comb}(D) \leq C N_k^{fun}(D)$, reflecting good decompositions $D= \sum b_S \Delta_S$.
The exact relation between these norms is unclear ; in particular, the constant $c$ satisfying $N_k^{comb}(D)\leq c N_k^{fun}(D)$ depends in a non-trivial way on $k$ and $\#R$, see the discussion in section \ref{sec_better_decompositions}.

\subsection{Other points of views}

We consider once again the primitives $I^kD$ of a distribution $D = \sum_{\beta} a_{\beta} \delta_{\beta}$ which are supported on $\mathbb{R}^+$.
Explicitely, $\displaystyle I^kD(t) = \sum_{\beta\leq t} a_{\beta} \frac{(t-\beta)^{k-1}}{(k-1)!}$ for $k\geq 1$.

\begin{prop}
\label{prop_IkD}
If $D$ is a distribution of order $k$, then $I^kD$ is a $\mc{C}^{k-1}$ function whose support is contained in $[\beta_0, \beta_n]$.
The value of $D$ on a test function $\varphi\in \mc{C}^k(I)$ is 
\[
\langle D, \varphi \rangle = (-1)^k \int_I (I^kD)(t) \varphi^{(k)}(t)dt.
\]
The functional norm of $D$ is
\[
N_k^{fun}(D) = \int_{I} | I^kD (t) | dt.
\]
\end{prop}

\begin{proof}
Suppose $R\subset [a,b]\subset I$.
Making $p$ integrations by parts on $[a,b]$ for $p\leq k$, we get 
\[
\langle D, \varphi \rangle = (-1)^p\int_{t=a}^b (I^pD)(t)\varphi^{(p)}(t)dt + \sum_{q=1}^p (-1)^{q+1}(I^qD)(b)\varphi^{(q-1)}(b).
\]
Taking $p=1$ and $\varphi=1$ we find $(ID)(b)=0$ ; for $p=2$ and $\varphi=t$ we get $(I^2D)(b)=0$, etc...
In fact $I^pD(b)=0$ for all $p\leq k$, which gives the first result.
The expression of $N_k^{fun}(D)$ is obtained as usual taking for $\varphi^{(k)}$ continuous approximations of $|I^kD| / I^kD$ where it is not zero.
\end{proof}

\begin{cor}
If $\#R=n+1$ and $\Delta_R$ is the normalized Vandermonde distribution supported on $R$, then $I^n \Delta_R$ is a positive function of class $\mc{C}^{n-1}$ supported by $\mr{Conv}(R)$.
In particular, $N_n^{fun}(\Delta_R) = 1$.
\end{cor}

\begin{proof}
Since $\Delta_R$ is real, we see by Propositions \ref{prop_majoration_vandermonde} and \ref{prop_IkD} that $\int_I|I^k \Delta_R(t)|dt \leq 1$.
On the other hand,
\[
\int_I (I^n \Delta_R)(t) dt = \langle \Delta_R, t^n/n! \rangle = 1 \geq \int_I |I^n \Delta_R(t)|dt.
\]
This can only be possible if $I^n \Delta_R$ is positive and $N_n^{fun}(\Delta_R)=1$.
\end{proof}

We can also study discrete distributions with their associated hyperfunctions.
The hyperfunction associated to $D = \sum_{\beta\in R} a_{\beta} \delta_{\beta}$ is 
\[
h(p) = \sum_{\beta\in R} \frac{a_{\beta}}{p-\beta}.
\]
If $\varphi$ is an analytic test function on $I$, which can be extended to a neighborhood $U$ of $I$ in $\mathbb{C}$, and $\gamma$ is a path in $U$ circling once around $I$,
\[
\langle D, \varphi \rangle = \frac{1}{2i\pi} \int_{p\in \gamma} h(p)\varphi(p)dp.
\]

\begin{prop}
\label{prop_vandermonde_hyperfunction}
The hyperfunction $h_R$ associated to the Vandermonde distribution $D_R$ is
\[
h_R(p) = \prod_{\beta\in R} \frac{1}{p-\beta}.
\]
\end{prop}

\begin{proof}
We know that $h_R(p) = \frac{P(p)}{\prod (p-\beta)}$ for some polinomial $P$ of degree $\mr{deg}(P)\leq n$.
If the test function $\varphi$ is polinomial, the value of the integral $\int_{\gamma} h_R(p) \varphi(p) dp$ can be computed as a residue at infinity.

The expression $h(p)$ in the statement can be written near infinity 
\[
h(p) = \prod_{\beta\in R} \frac{(p^{-1})^{n+1}}{1- \beta p^{-1}}.
\]
The residues of $p^kh(p)dp$ at infinity coincide with the values of $\langle D_R, t^k \rangle$ for $k\leq n$, showing that $h=h_R$.
\end{proof}

\subsection{Decomposing a distribution on a Vandermonde basis}

\begin{lem}
\label{lem_decomposition_distribution}
Consider a distribution $D$ supported on a finite set $R=\{\beta_0,\ldots,\beta_n\}$ contained in an interval $I$ of length $L$, denote by $\Delta_i = \Delta_{\{\beta_0,\ldots,\beta_i\}}$ and write 
\[
D = b_0 \Delta_0 + \ldots + b_n \Delta_n.
\]
Suppose that $D$ is of order $k$ and $N_k^{fun}(D) \leq C$.
Then $b_0 = \ldots = b_{k-1}=0$ and $|b_i|\leq C L^{i-k}/(i-k)!$ for all $i\geq k$.
\end{lem}

In particular, we have $N_{\geq k}^{comb}(D) \leq \mr{e} N_{k}^{fun}(D)$ if we choose the parameter $\rho=L$ in the definition of combinatorial norm.

\begin{proof}
Consider the polynomial $P_{i+1}(t) = (t- \beta_0) \ldots (t- \beta_i)$.
Since $P_i$ has degree $i$, we have $\Delta_j(P_i)=0$ for $j>i$ ; by its very definition, we have $\Delta_j(P_i) = 0$ for $j<i$.
Thus $D(P_i) = i!b_i$ for all $i=0,\ldots,n$.
Now, for every $i\geq k$, we can compute the k-th derivative
\[
P_i^{(k)} = \sum_{0\leq j_0 < \ldots < j_{i-k-1}\leq i-1} k!(t-\beta_{j_0})\ldots (t- \beta_{j_{i-k-1}}).
\]
Since $L$ is the length of the interval $I$, then $|t-\beta_j|\leq L$ for all $j$, and thus
\[
\| P_i^{(k)}\|_{I} \leq k! C_{i}^k L^{i-k} = i! \frac{L^{i-k}}{(i-k)!}.
\]
We deduce that $|b_i|\leq N_k^{fun}(D) \|P_i^{(k)}\|_{I}/i! \leq C\:L^{i-k}/(i-k)!$
\end{proof}

One can also try to decompose $D$ on the basis $\Delta_{R_j}$ where $R_j = \{\beta_j,\ldots,\beta_{j+k}\}$.
An easy way to do this is to use the latter lemma with Lemma \ref{lem_decomposition_vandermonde} to decompose each $\Delta_i$ with $i >k$ on the basis $\Delta_{R_j}$.
We will obtain a decomposition $D = \sum c_j \Delta_{R_j}$, but the coefficient $c_j$ will have a factor $\prod (\beta_q - \beta_p)^{-1}$, for some pairs $(\beta_p,\beta_q)$ where $\beta_p$ is on the left of $R_j$ and $\beta_q$ on the right.
We do not have lower estimates for these $\beta_p-\beta_q$ which are good enough to do something useful in this direction.

\subsection{Degenerated discrete distributions}
\label{sec_degenerated}

Consider some finite multiset $R = \{ (\beta_0,m_0),\ldots, (\beta_n,m_n)\}$ where $m_k>0$ is the multiplicity of $\beta_k$.
The distributions supported on $R$ will be those of the form 
\[
D = \sum_{k=0}^n \sum_{r=0}^{m_k-1} (-1)^ra_{(\beta_k,r)} \delta_{\beta_k}^{(r)}.
\]
The cardinality of $R$ being $\#R = N+1 = m_0 + \ldots +m_n $, we can once again define a Vandermonde distribution $D_R$ as the solution of the system 
\[
\left\{
\begin{aligned}
D_R(1)&=0\\
D_R(t)&=0\\
&\vdots\\
D_R(t^N)&=1
\end{aligned}
\right. \Leftrightarrow \left\{
\begin{aligned}
&\sum_{k=0}^n a_{(\beta_k,0)} = 0\\
&\sum_{k=0}^n \beta_k a_{(\beta_k,0)} + a_{(\beta_k,1)} =0\\
&\quad\quad\quad\vdots\\
&\sum_{k=0}^n \sum_{r=0}^{m_k-1} \beta_k^{N-r} \frac{N!}{(N-r)!}a_{(\beta_k,r)} =1. 
\end{aligned}
\right.
\]

\begin{lem}
\label{lem_degenerated_distributions}
Consider a continuous family of multisets $R(\lambda) = \{ \beta_0(\lambda), \ldots, \beta_N(\lambda)\}$.
Then $\lambda \mapsto D_{R(\lambda)}$ is continuous for the norm $N_N^{fun}$.
\end{lem}

\begin{proof}
Choose intervals $I\supset [a,b]\supset R$ and a path $\gamma$ circling once around $I$, which we can suppose to be the border of an open set $U\subset \mathbb{C}$.

Proposition \ref{prop_IkD} and \ref{prop_vandermonde_hyperfunction} are still valid for degenerated distributions, as one easily checks.
Following Proposition \ref{prop_IkD}, we want to prove that $\lambda \mapsto I^N D_{R(\lambda)}$ is continuous in the $L^1$ norm.
Using Proposition \ref{prop_vandermonde_hyperfunction}, we see that $\lambda \mapsto h_{R(\lambda)}$ is continuous for the norm $\mr{sup}_{p\in \gamma} |h(p)|$ (where $h_{R(\lambda)}$ is the hyperfunction associated to $D_{R(\lambda)}$).

However the hyperfunction corresponding to the function $I^N D_{R(\lambda)}$ is $I^N h_{R(\lambda)}$ where $Ih(p)$ can be chosen to be any primitive of $h$ (they all give the same hyperfunction).
We will choose $I^kh$ to be the primitives given by Cauchy's integral formula on $\gamma$, for example 
\[
Ih(z) = \frac{1}{2i\pi} \int_{\gamma} h(p) \mr{log}(p-z)dp,
\]
for $z\in U\setminus [a,+\infty[$.
Here $\mr{log}(p-z)$ stands for the analytic continuation in $p$ and $z$ of the canonical determination of the logarithm defined for $p\in \gamma\setminus [a,+\infty[$ and $z\in I\setminus [a,+\infty[$.
The integration is always done on the path $\gamma\setminus [a,+\infty[$.

Using this formula, we get primitives $I^Nh_{R(\lambda)}$ depending continuously on $\lambda$ ; they extend continuously to $[a,b]$ from above and from below, and if $t\in [a,b]$,
\[
I^N D_{R(\lambda)} (t) = \lim_{z\to t+0^+i} I^Nh_{R(\lambda)}(z) - \lim_{z\to t+0^-i} I^Nh_{R(\lambda)}(z).
\]
We see that $\lambda \mapsto I^ND_{R(\lambda)}$ is in fact continuous for the uniform norm.
\end{proof}

This lemma justifies that any degenerated distribution can be approximated to any precision by a discrete distribution as described in the former subsections.
In particular, border terms of an integration by parts can be taken to be a non-degenerated distribution.

\subsection{Obtaining better decompositions ?}
\label{sec_better_decompositions}

To prove Theorem \ref{thm_summation_by_packages}, we will use Lemma \ref{lem_decomposition_distribution} to decompose a general distribution of order $k$ as a sum of Vandermonde distributions.
This decomposition is not optimal in that it involves distributions of order $p > k$, which will change the shape of the neighborhood on which the summation by packages converges.

It might be usefull to discuss here the difficulties to obtain better decompositions (ideally decomposition of $D$ as a sum of Vandermonde distributions of the same order as $D$).
We use the same notations in this discussion as in all this section.

For simplicity, we can consider real distributions $D$ of order $k$ normalized by $D(t^k/k!) = 1$.
The set of such distributions is an affine subspace $A$ spanned by the distributions $\Delta_{S}$ with $\#S=k+1$.
The set of distributions $D$ with $N_k^{comb}(D)=1$ is the convex hull $A^1_{comb}$ of the $\Delta_{S}$, which we want to compare to the convex set $A^1_{fun}$ of distributions $D$ with $N_k^{fun}(D)=1$.
The difference between these sets will reflect the difference between the norms : of course $A^1_{comb}\subset A^1_{fun}$.
Note that all distributions in $A^1_{fun}$ deserve to be considered as discrete approximations of the $k$-th derivative of the dirac, they all play the same role as one of the $\Delta_S$.

From Proposition \ref{prop_IkD}, we see that in the real case, $A^1_{fun}$ is the set of distributions $D\in A$ such that $I^kD$ is positive.
We immediately deduce a general construction to produce elements in $A^1_{fun}\setminus A^1_{comb}$ : take three sets $S_1,S_2,S_3\subset R$ of cardinality $k+1$, such that the convex hull $\mr{conv}(S_3)$ is included in the union $\mr{conv}(S_1)\cup \mr{conv}(S_2)$.
We get an inclusion for the supports of the corresponding distributions $\mr{supp}(I^k \Delta_{S_3}) \subset \mr{supp}(I^k \Delta_{S_1}+I^k \Delta_{S_2})$.
A combination $a_1 \Delta_{S_1} + a_2 \Delta_{S_2} - a_3 \Delta_{S_3}$ (with $a_i>0$ and $a_3$ not too big) will still be positive, and after renormalization, we get a distribution in $A^1_{fun}\setminus A^1_{comb}$.

What these examples show is that the family $(\Delta_S)$ is not the natural family on which to decompose a distribution : the natural family should be the set of extremal elements of $A^1_{fun}$.
However, the construction above has a lot of degree of freeness, the set of extremal elements of $A^1_{fun}$ does not even seem to be finite.
We can certainly obtain a theoretical result by decomposing distributions on the set of extremal distributions, but stating a general decomposition on a family that we don't understand will be no better than the general Theorem \ref{thm_weak_summation}.
Another possible direction to obtain better decompositions would be to estimate directly how an extremal distribution decomposes on the family $(\Delta_S)$ (which doesn't seem anything obvious).

\section{Summation by packages of irrational series}
\label{sec_summation}

We are ready to prove the main theorem.

\begin{thm}
\label{thm_summation_by_packages}
Consider an irrational series $g(w) = \sum_{\beta\in R} a_{\beta} \mr{e}^{\beta w}$ supported on a closed discrete set $R\subset \mathbb{R}^+$, defined and bounded in a logarithmic half-plane $H_{a,k}$.
We will suppose that $R$ has linear density : there exist constants $\mu,\nu$ such that for any interval $I$ of length $L$, $\# (R\cap I) \leq \mu\: L\: \mr{sup}(I) + \nu\: L$.

There exist finite sets $R_n$ with $\mr{inf}(R_n) = \beta_n$, $\#R_n = N_n+1$, coefficients $b_n\in \mathbb{C}$, a radius $r\in ]0,1[$ and a constant $c\in \mathbb{R}$ such that 
\[
\begin{aligned}
&g(w) = \sum_n b_n \langle \Delta_{R_n}, \mr{e}^{tw} \rangle,\\
& \sum_n |b_n| r^{\beta_n} < \infty,\\
& N_n \leq (\mu + 2k) \beta_n + c.
\end{aligned}
\]
\end{thm}

Recall that by Corollary \ref{cor_majoration_vandermonde}, $|\langle \Delta_{R_n}, \varphi\rangle| \leq \sqrt{2}\: \mr{sup}_I|\varphi^{(N_n)}|$ if $I = \mr{Conv}(R_n)$.
In particular, when $\mr{Re}(w)<0$,
\[
| \langle \Delta_{R_n}, \mr{e}^{tw} \rangle | \leq \sqrt{2}\: |w^{N_n} \mr{e}^{\beta_n w} |.
\]
Using the estimate of Lemma \ref{lem_cle}, the sum $\sum b_n \langle \Delta_{R_n}, \mr{e}^{tw} \rangle$ converges by packages in some logarithmic neighborhood of $-\infty$.

\begin{proof}
For technical reasons, we will suppose $k\geq 1$ in the proof : since the sets $H_{a,k}$ are steeper when $k$ is big, this hypothesis is not restrictive.

We begin with Proposition \ref{prop_admissible} to obtain an admissible sequence $t_n = \frac{n-3.5}{k}$ for $n$ big enough for which $I^{\Delta}_{(t_n)}(D,\mr{e}^{tw})$ converges.
Using Proposition \ref{prop_DIPP_summation}, this convergence gives a sequence of degenerated distributions $D_n$ supported on $[t_n,t_{n+1}]$, of order $n$, whose support is a multiset $M_n$ of cardinality $\#M_n = \#(R\cap [t_n,t_{n+1}]) + 2n$ ; by Proposition \ref{prop_DIPP_bounded} the DIPP is bounded by some $a\in \mathbb{R}$ (which we can suppose to be negative $a\leq 0$), which implies $C := \sum_n N_n^{fun}(D_n) \mr{e}^{a t_{n+1}} < \infty$.
In particular, we have $N_n^{fun}(D_n) \leq C \mr{e}^{-a t_{n+1}}$.

Since $R$ has linear density, we have a bound $\#M_n \leq \left(\mu/k + 2 \right) n + \nu$.
Decomposing the distribution $D_n$ with Lemma \ref{lem_decomposition_distribution}, we get
\[
D_n = \sum_{p\geq n} b_{n,p} \Delta_{n,p},
\]
where $\Delta_{n,p}$ is a normalized Vandermonde distribution supported on a multiset $M_{n,p}\subset M_n$ of cardinal $\#M_{n,p} = p$, and
\[
\begin{aligned}
|b_{n,p}| &\leq C \mr{e}^{-a t_{n+1}} \frac{(1/k)^{p-n}}{(p-n) !}\\
&\leq C \frac{\mr{e}^{-a t_{n+1}}}{(p-n)!}.
\end{aligned}
\]
Introduce $\beta_{n,p} = \mr{inf}(M_{n,p})$ as in the statement, we have $\beta_{n,p} \geq t_n$.
Thus
\[
\begin{aligned}
\sum_{p\geq n\geq 0} |b_{n,p}| \mr{e}^{(a-1)\beta_{n,p}} &\leq \sum_{p\geq n\geq 0} C \frac{\mr{e}^{a(t_n-t_{n+1})}\mr{e}^{-t_n}}{(p-n)!}\\
&\leq \sum_{n\geq 0} C\mr{e} \mr{e}^{-a/k}\mr{e}^{-t_{n}}\\
& < \infty.
\end{aligned}
\]
The cardinal of $M_{n,p}$ is bounded by 
\[
\begin{aligned}
\#M_{n,p} &\leq \#M_n\\
&\leq \left( \frac{\mu}{k} + 2 \right) n + \nu\\
&\leq \left( \frac{\mu}{k} + 2 \right) (k t_n + 3.5) + \nu \\
&\leq ( \mu + 2k ) \beta_{n,p} + 3.5 \left( \frac{\mu}{k}+2 \right) + \nu.
\end{aligned}
\]

This gives the result if we accept multisets $R_{n,p}$ and degenerated discrete distributions $\Delta_{R_{n,p}}$.
The derivative of diracs only appear as border terms at the cutting points $t_n$ (once for $D_n$ and then in $D_{n+1}$ with a minus sign : border terms need to cancel out in the sum).
We can use Lemma \ref{lem_degenerated_distributions} to aproximate these border terms with sums of diracs (with arbitrary precision) in order to obtain the statement of the theorem (involving usual sets $R'_{n,p}$ and non-degenerated discrete distributions $\Delta_{R'_{n,p}}$).
\end{proof}

The theorem can be applied for irrational series $g(w) = \sum_{\beta\in R_{\alpha}} a_{\beta} \mr{e}^{\beta w}$ where $R_{\alpha} = \mathbb{N} + \alpha^{-1} \mathbb{N}$.
Indeed, $\#(R\cap ]\beta-1,\beta]) = 1+ \lfloor {\alpha \beta} \rfloor$ : with the vocabulary of Theorem \ref{thm_summation_by_packages}, the set $R_{\alpha}$ has linear density with $\mu = \alpha$.

The theorem can also be applied for usual power series $g(w) = \sum_{\beta\in \mathbb{N}} a_{\beta} \mr{e}^{\beta w}$.
The estimates do not allow to prove that such a series converges normally directly, but we can prove it differently.
Suppose $g$ is defined in a logarithmic half-plane $H$, and introduce the translation $\tau(w) = w + 2i\pi$.
We see that $g\circ \tau$ is an irrational series with the same formal development so that $g\circ\tau = g$, proving that $g(w) = f(\mr{e}^w)$ for some holomorphic function $f$ defined on the quotient $H/\tau$, that is, a normally convergent power series $f(z) = \sum a_n z^n$.

\appendix

\section{Some neighborhoods of $-\infty$}

\subsection{Logarithmic neighborhoods}
\label{appendix_logarithmic}

Recall that the logarithmic neighborhood $H_{a,k}$ is defined as
\[
H_{a,k} = \left\{ x+iy \in \mathbb{C} \:|\: x + \frac{k}{2}\mr{log}(x^2+y^2) < a \right\}.
\]
The following lemma is used throughout the paper.

\begin{lem}
\label{lem_cle}
For any $w\in H_{a,k}$, any $\beta>0$ and $p\leq k \beta$, we have
\[
| w^p\mr{e}^{\beta w} | \leq \mr{e}^{a \beta}.
\]
\end{lem}

\begin{proof}
Just note that $\displaystyle \mr{log} \left( |w^p \mr{e}^{\beta w} | \right) = \frac{p}{2}\mr{log}(x^2+y^2) + \beta x \leq \beta \left(x + \frac{k}{2}\mr{log}(x^2+y^2) \right)$.
\end{proof}

\subsection{Straight half-planes}

Consider a half-plane $\mb{H} = \{ \mr{Re}(w) < a\}$ and a bounded function $g\in \mc{O}(\mb{H})$ with formal development $\hat{g}(w) = \sum_{\beta\in R} a_{\beta} \mr{e}^{\beta w}$.

In \cite[Section 2.7]{thom_irrational1}, we proved that the hyperfunction $h=\mc{L}g$ satisfied an estimate 
\[
|h(p)| \leq C \frac{\mr{e}^{-w_0 \mr{Re}(p)}}{|\mr{Im}(p)|}
\]
for $\mr{Re}(p) >0$.
Since the coefficient $a_{\beta}$ is the residue of $h$ at $p=\beta$, we get the estimate $|a_{\beta}|\leq C \mr{e}^{-w_0 \beta}$.
If the index set $R$ is not too dense (with polynomial density for example), this proves that the series $g(w) = \sum_{\beta \in R} a_{\beta} \mr{e}^{\beta w}$ converges normally in some half-plane.
If $R$ has more than exponential density, we might need another notion of convergence, maybe some kind of convergence by packages.

\subsection{Quadratic domains}

In the classical works about Dulac germs, the neighborhoods considered are the standard quadratic domains (cf. \cite{ilyashenko_centennial}).
Denote once again $\mb{H} = \{ \mr{Re}(w) < a \}$ for some $a\leq 0$, put $\Phi_C(w) = w - C \sqrt{1+ w^2}$ for $C>0$ (we use the usual determination of the square root : when $\mr{Re}(w)<0$, we have $\sqrt{1+w^2}\sim -w$ when $|w|\to \infty$).
The quadratic domain $\Omega_C$ is defined as $\Omega_C = \Phi_C(\mb{H})$ (for coherence with the rest of the article, we adapted the definition to consider $\Omega_C$ as a neighborhood of $-\infty$, we also introduced the parameter $a$ for flexibility).
For curiosity, we can try to understand to which notion of convergence these neighborhoods correspond.

The shape of $\Omega_C$ can be deduced from the assymptotics of its border $\partial_{}\Omega_C = \Phi_C(a + i \mathbb{R})$.
Note that 
\[
\begin{aligned}
| (1+C)w - \Phi_C(w)| & = C | w + \sqrt{1+w^2}|\\
&\leq \frac{C}{|w - \sqrt{1+w^2}|}.
\end{aligned}
\]
Since $(a+it) - \sqrt{1+ (a+it)^2} \sim 2 (a+it)$ when $t\to \pm\infty$, we see that $| (1+C)w - \Phi_C(w)|$ is bounded for $w\in a+i \mathbb{R}$ and $\Omega_C$ contains a straight half-plane.
Thus, the natural notion of convergence for functions bounded on quadratic domains is that of normal convergence.

\bibliography{mybib}{}

@incollection{ecalle_small_denominators,
 author = {Ecalle, Jean},
 title = {Compensation of small denominators and ramified linearisation of local objects},
 booktitle = {Complex analytic methods in dynamical systems. Proceedings of the congress held at Instituto de Matem\'atica Pura e Aplicada, IMPA, Rio de Janeiro, Brazil, January 1992},
 pages = {135--199},
 year = {1994},
 publisher = {Paris: Soci{\'e}t{\'e} Math{\'e}matique de France},
 language = {English},
 url = {smf4.emath.fr/Publications/Asterisque/1994/222/html/smf_ast_222_135-199.html},
 zbMATH = {671754},
 Zbl = {0810.58036}
}

@incollection{ecalle_growth_scale,
 author = {Ecalle, Jean},
 title = {The natural growth scale},
 booktitle = {Algebraic combinatorics, resurgence, moulds and applications (CARMA). Volume 1},
 isbn = {978-3-03719-204-7; 978-3-03719-704-2},
 pages = {93--223},
 year = {2020},
 publisher = {Berlin: European Mathematical Society (EMS)},
 language = {English},
 doi = {10.4171/204-1/4},
 zbMATH = {7219058},
 Zbl = {1494.30004}
}

@article{emmr_cycles_limites,
 author = {Ecalle, Jean and Martinet, Jean and Moussu, Robert and Ramis, Jean-Pierre},
 title = {Non-accumulation des cycles-limites. {I}. {II}. ({Nonaccumulation} of limit- cycles. {I}. {II}.)},
 fjournal = {Comptes Rendus de l'Acad{\'e}mie des Sciences. S{\'e}rie I},
 journal = {C. R. Acad. Sci., Paris, S{\'e}r. I},
 issn = {0764-4442},
 volume = {304},
 pages = {375--377, 431--434},
 year = {1987},
 language = {French},
 zbMATH = {3996689},
 Zbl = {0615.58011}
}

@article{ilyashenko_centennial,
 author = {Ilyashenko, Yu.},
 title = {Centennial history of {Hilbert}'s 16th {Problem}},
 fjournal = {Bulletin of the American Mathematical Society. New Series},
 journal = {Bull. Am. Math. Soc., New Ser.},
 issn = {0273-0979},
 volume = {39},
 number = {3},
 pages = {301--354},
 year = {2002},
 language = {English},
 doi = {10.1090/S0273-0979-02-00946-1},
 zbMATH = {1756730},
 Zbl = {1004.34017}
}

@article{ilyashenko_limit_cycles,
 author = {Il'yashenko, Yu. S.},
 title = {Finiteness theorems for limit cycles},
 fjournal = {Translations. Series 2. American Mathematical Society},
 journal = {Transl., Ser. 2, Am. Math. Soc.},
 issn = {0065-9290},
 volume = {148},
 pages = {55--64},
 year = {1991},
 language = {English},
 doi = {10.1090/trans2/148/06},
 zbMATH = {12934},
 Zbl = {0768.34017}
}

@book{morimoto_hyperfunctions,
 author = {Morimoto, Mitsuo},
 title = {An introduction to {Sato}'s hyperfunctions. {Transl}. from the {Japanese} by {Mitsuo} {Morimoto}},
 fseries = {Translations of Mathematical Monographs},
 series = {Transl. Math. Monogr.},
 issn = {0065-9282},
 volume = {129},
 isbn = {0-8218-4571-3},
 year = {1993},
 publisher = {Providence, RI: American Mathematical Society},
 language = {English},
 zbMATH = {482841},
 Zbl = {0811.46034}
}

@article{sato_hyperfunctions1,
 author = {Sato, Mikio},
 title = {Theory of hyperfunctions. {I}},
 fjournal = {Journal of the Faculty of Science. Section I},
 journal = {J. Fac. Sci., Univ. Tokyo, Sect. I},
 issn = {0368-2269},
 volume = {8},
 pages = {139--193},
 year = {1959},
 language = {English},
 zbMATH = {3143096},
 Zbl = {0087.31402}
}

@article{sato_hyperfunctions2,
 author = {Sato, Mikio},
 title = {Theory of hyperfunctions. {II}},
 fjournal = {Journal of the Faculty of Science. Section I},
 journal = {J. Fac. Sci., Univ. Tokyo, Sect. I},
 issn = {0368-2269},
 volume = {8},
 pages = {387--437},
 year = {1960},
 language = {English},
 zbMATH = {3158986},
 Zbl = {0097.31404}
}

@unpublished{thom_irrational1,
author = {Thom, O.},
title = {Irrational series {I}, {L}aplace transform in a neighborhood of $-\infty$},
note = {Disponible on Arxiv}
}
\bibliographystyle{acm}

\vfill
\textsc{Universidade Federal Fluminense - Instituto de Matemática e Estatística, Niterói, Brasil}

\textit{Email :} olivier\_thom@id.uff.br

\end{document}